\documentclass[12pt,twoside]{article}
\usepackage{a4}
\usepackage{amssymb,amsmath,amsthm,latexsym}
\usepackage{amsfonts}
\usepackage{amsfonts}
\usepackage{graphicx}
\usepackage{nicematrix}
\usepackage{multirow}

\newtheorem{theorem}{Theorem}[section]

\newtheorem{corollary}[theorem] {Corollary}
\newtheorem{definition}[theorem]{Definition}
\newtheorem{example}[theorem]{Example}
\newtheorem{lemma} [theorem]{Lemma}

\newtheorem{proposition}[theorem]{Proposition}
\newtheorem{remark}[theorem]{Remark}

\usepackage{longtable}
\usepackage{caption}

\vspace{5cm}

\begin{document}

\label{'ubf'}  
\setcounter{page}{1}                                 

\markboth {\hspace*{-9mm} \centerline{\footnotesize \sc
	\scriptsize { On the characterization of Eulerian $es$-splitting $p$-matroids}}
}
{ \centerline                           {\footnotesize \small  
	 Uday Jagadale$^1,$ Prashant Malavadkar$^2,$ Sachin Gunjal$^3,$ M.M.Shikare$^4,$B.N. Waphare$^5$                                                  } \hspace*{-9mm}              
}

\begin{center}
{ 
	{\Large \textbf { \sc On the characterization of Eulerian $es$-splitting $p$-matroids}
	}
	\\

	\medskip

	{\sc Uday Jagadale$^1,$ Prashant Malavadkar$^{2 *},$ Sachin Gunjal$^3,$ and M.M.Shikare$^4$  }\\
	{\footnotesize  $^{1,2,3}$Dr. Vishwanath Karad MIT World Peace University,
		Pune 411 038,	India.\\
	$^{4}$Savitribai Phule Pune University, Pune-411007, India }\\
	{\footnotesize e-mail: {\it 1.uday.jagdale@mitwpu.edu.in, 2.prashant.malavadkar@mitwpu.edu.in, 3.sachin.gunjal@mitwpu.edu.in, 4.mmshikare@gmail.com}}
}
\end{center}

\thispagestyle{empty}

\hrulefill

\begin{abstract}  
{\footnotesize 
\noindent The $es$-splitting operation on binary bridge-less matroids never produces an Eulerian matroid. But for matroids representable over $GF(p),(p>2),$ called $p$-matroids, the $es$-splitting operation may yield Eulerian matroids. In this work, we introduce $es$-splitting operation for $p$-matroids and characterize a class of $p$-matroids yielding Eulerian matroids after the $es$-splitting operation. Characterization of circuits, and bases of the resulting matroid, after the $es$-splitting operation, in terms of circuits, and bases of the original matroid, respectively, are discussed. We also proved that the $es$-splitting operation on $p$-matroids preserves connectivity and 3-connectedness. Sufficient condition to obtain Hamiltonian $p$-matroid from Hamiltonian $p$-matroid under $es$-splitting operation is also provided.}
\end{abstract}
\hrulefill



\section{Introduction}

The $n$- line splitting operation for graphs was introduced by Slater \cite{slt1} and it was extended to binary matroids by Azanchiler \cite{AH1}. Properties of splitting, element splitting, and $es$ -splitting operations on binary matroids are extensively discussed by several authors \cite{AH2,DMS,Mills, rsw,ma, shika,ba}.
Malavadkar et al. \cite{mjg, elt} defined splitting and element splitting operations for simple and coloopless $p$-matroids, i.e. for matroids representable over $GF(p).$

\begin{definition}\label{def0}
	Let $M\cong M[A]$ be a $p$-matroid on ground set $E,$$\{a,b\} \subset E,$ and $\alpha\neq 0$ in $ GF(p)$. The matrix $A_{a,b}$ is constructed  from $A$ by appending an extra row to $A$ which has coordinates equal to $\alpha$ in the columns corresponding to the elements $a,$$b,$ and zero elsewhere. Define the splitting matroid $M_{a,b}$ to be the vector matroid  $M[A_{a,b}].$ The transformation of $M$ to $M_{a,b}$ is called  splitting operation.
\end{definition}

\begin{definition}\label{def1}
	
	Let $M\cong M[A]$ be a $p$-matroid on ground set $E,$$\{a,b\} \subset E,$ and $M_{a,b}$ be the corresponding splitting matroid. Let the matrix $A_{a,b}$ represents $M_{a,b}$ on $GF(p).$ Construct the matrix $A'_{a,b}$ from $A_{a,b}$ by adding an extra column to $A_{a,b},$labeled as $z,$ which has the last coordinate equal to $\alpha\neq 0$ and the rest are equal to zero. Define the element splitting matroid $M'_{a,b}$ to be the vector matroid  $M[A'_{a,b}]$. The transformation of $M$ to $M'_{a,b}$ is called element splitting operation.
	
\end{definition}

\noindent A matroid $L$ is a lift of the matroid $M,$ if there exists a matroid $N,$ and $X\subset E(N)$ such that $N/X=M,$ and $N \setminus X=L.$ If $|X|=1,$ then $L$ is called an elementary lift of $M.$ In the next result, Mundhe et al. \cite{mundhe} proved that for binary matroid $M$ its elementary lift $L$ and splitting matroid $M_T$ are isomorphic.
A study on lifted graphic matroids is carried out in \cite{CG, CW, FM}. 
\begin{lemma}
	Let $M$ and $L$ be binary matroids. Then $L$ is an elementary lift of $M$ if and only if $L$ is isomorphic to $M_T$ for some $T\subset E(M).$ 
\end{lemma}
\noindent Lemma 1.3 can be extended to $p$-matroids  by using the similar arguments used to prove it   in \cite{mundhe}. Thus a splitting matroid $M_{a,b}$ of $p$-matroid $M$ is an elementary lift of  $M.$ In this case the matroid $N$ is element splitting matroid $M'_{a,b}.$\\

\noindent In this paper, we define $es$-splitting operation on a simple and coloopless $p$-matroid, which is a single element extension of corresponding element splitting $p$-matroid, and study circuits, bases and connectivity of the resulting matroid.

\noindent A binary matroid is Eulerian if and only if it's dual matroid is bipartite. Equivalently, an Eulerian binary matroid can not have an odd cocircuit. The $es$-splitting matroid $M^e_{a,b}$of a bridge-less binary matroid $M$ always contains an odd cocircuit. Therefore $M^e_{a,b}$ can never be Eulerian. But when we consider matroids representable over $GF(p),(p>2),$ the $es$-splitting matroid of it contains an odd cocircuit and still can be Eulerian. In the present paper, we provide necessary and sufficient condition for an $es$-splitting $p$-matroid $M^e_{a,b}$ to be Eulerian. \\

\noindent A sufficient condition to obtain Hamiltonian $p$-matroid from Hamiltonian $p$-matroid under $es$-splitting operation is also provided. For undefined and standard terminologies, refer to \cite{ox}. 
\section{$es$-splitting operation on $p$-matroids}
The present section generalizes the $es$-splitting operation to a $p$-matroid $M,$ and characterizes circuits and bases of the $es$-splitting matroid.
\begin{definition}\label{d3}
	Let $M\cong M[A]$ be a $p$-matroid on ground set $E$ and $\{a,b\}\subset E$ with $e\in \{a,b\}.$ Let $A'_{a,b}$ represents the element splitting matroid $M'_{a,b}$ on $GF(p).$ Construct the matrix $A^e_{a,b}$ by adjoining an extra column $\gamma$ to the matrix $A'_{a,b}$ where $\gamma$ is subtraction$(e-z)$ of two column vectors corresponding to the elements $e$ and $z.$ Denote $M^e_{a,b}\cong M[A^e_{a,b}].$ The shift of $M$ to $M^e_{a,b}$ is called an $es$-splitting operation. We call the matroid $M^e_{a,b}$ as $es$-splitting $p$-matroid.  
\end{definition}

\begin{remark}
	For $p=2,$ the above definition coincides with the $es$-splitting operation on binary matroids \cite{AH1}.  
\end{remark}
\begin{remark}
	rank$(A)<$ rank$(A^e_{a,b})=$ rank$(A)+1.$ If the rank functions of $M$ and $M^e_{a,b}$ are denoted by $r$ and $r',$respectively, then $r(M) < r'(M^e_{a,b})= r(M) + 1.$
\end{remark}

\noindent Let $C_k=\{u_1, u_2,\ldots, u_l : u_i \in E, i=1,2,\dots,l \}$  be a circuit of $M$ for some positive integers $k,$ $l$ and $|C_k \cap \{a,b\}|=2.$ Suppose, without loss of generality, $u_1=a$,$u_2=b$. If there are non-zero constants $a_1^k, a_2^k, ..., a_l^k$  in $GF(p)$ such that $\sum_{i=1}^{i=l} a_i^k u_i \equiv 0 (mod~p)$ and $a_1^k + a_2^k \equiv 0 (mod~p),$ then we call $C_{k}$ a $p$-circuit of $M.$ However, if $|C_k \cap \{a,b\}|=1$ or $|C_k \cap \{a,b\}|=2$ and $a_1^k + a_2^k\not\equiv 0 (mod \: p),$ then we call $C_{k}$ an $np$-circuit of $M$.

\noindent We denote $\mathcal {C}_0=\{C\in \mathcal {C}(M): C$ is a $p$-circuit or $C\cap \{a,b\}=\phi\}$.

\begin{theorem}\label{circuits}
	Let $M=(E,\mathcal{C})$ be a $p$-matroid, and $M'_{a,b}$ be its element splitting $p$-matroid. $e\in \{a,b\}$ and $z, \gamma \notin E$. Then $\mathcal{C}(M^e_{a,b})= \mathcal{C}(M'_{a,b}) \cup \mathcal{C}_4 \cup \mathcal{C}_5 \cup \mathcal{C}_6 \cup \mathcal{C}_7 \cup \mathcal{C}_8\cup \Delta,$ where 	
	\begin{description}
		
		\item $\mathcal {C}_4=\{C\cup \{e, \gamma\}: C$ is an $np$-circuit and $e\notin C\}$;
		
		\item $\mathcal {C}_5=\{(C\setminus e)\cup \gamma: C$ is an $np$-circuit, $e\in C$ and $a\notin C\}$;
		
		\item $\mathcal {C}_6=\{(C\setminus e)\cup \{z, \gamma \}: C$ is a $p$-circuit and $e\in C\}$;

		\item $\mathcal {C}_7=\{C \cup \gamma : C$ is an $np$-circuit containing $e$ and $a$\};
		
		\item $\mathcal {C}_8=\{(C\setminus e) \cup \{z, \gamma\} : C$ is an $np$-circuit containing $e$ and $a$\} and
		
		\item $\Delta=\{e, z, \gamma\}.$
		
	\end{description}	
\end{theorem}
\begin{proof}
	By the Definition \ref{d3} and $M'_{a,b}=M^e_{a,b}|_{E\cup z},$ the inclusion  $ \mathcal{C}(M'_{a,b}) \cup \mathcal{C}_4 \cup \mathcal{C}_5 \cup \mathcal{C}_6 \cup \mathcal{C}_7 \cup \mathcal{C}_8\cup \Delta \subset \mathcal{C}(M^e_{a,b})$ follows. 
	
	\noindent To prove the other inclusion, assume $C \in \mathcal{C}(M^e_{a,b})$ and $\gamma \notin C.$ Then by Theorem 2.4 of \cite{elt}, $C \in \mathcal{C}(M'_{a,b}).$ Now assume that $\gamma \in C.$\\
	$\mathbf{Case~1:}$ $z \in C.$ Note that if $e\in C,$ then $C=\Delta=\{e, z, \gamma\}$. Suppose $e \notin C.$ If $C_1= C\setminus\{z,\gamma\} $ is a dependent set of $M$ containing an $np$-circuit, say $C_2.$ Then $C_2 \cup z$ is a circuit of $M^e_{a,b}$ contained in $C,$ which is not possible. Therefore $C_1$ is an independent set of $M.$
	Since the every coordinate of $e$ is equal to the corresponding coordinate of $\gamma,$ and all the coordinates of $z$ except the last one are zero, the set $C_1 \cup e $ is a dependent set of $M.$ Moreover $C_1 \cup e $ is a circuit of $M,$ otherwise $C_1 \cup e $ contains an $np$-circuit, say $C_3.$ And $C_3 \cup z $ is circuit of $M^e_{a,b}$ contained in $C$ which is not possible. If $C_1 \cup e $ is a $p$-circuit, then $C \in \mathcal {C}_6.$ And if $C_1 \cup e $ is an $np$-circuit, then $C \in \mathcal {C}_8.$

	\noindent $\mathbf{Case~2:} $ $z\notin C$\\
	\noindent$\mathbf{Subcase~(i):}$  $e \notin C\setminus \gamma.$  Then observe that $a$ is also not an element of $C.$  And by similar argument as in Case 1, $(C\setminus \gamma)\cup e$ is an $np$-circuit of $M.$ Therefore $C\in \mathcal{C}_5.$ \\
	\noindent $\mathbf{Subcase~(ii):}$ $e \in C\setminus \gamma$. Note that $C\setminus\gamma$ is a dependent set of $M,$ containing $a$ and one of the following  cases will occur.\\
	$\mathbf{(a)}$ If $C\setminus \gamma$ is a circuit of $M,$ then it must be an $np$-circuit. Therefore $C\in \mathcal{C}_7.$\\
	\noindent $\mathbf{(b)}$ If $C\setminus \gamma$ is not a circuit of $M,$ and $C\setminus \gamma=C_{np}\cup e,$ for some $np$-circuit $C_{np}$ excluding $e.$ Then  $C_{np} \cup \{e, \gamma\}=C.$ In this case $C\in \mathcal{C}_4.$ \\
	\noindent $\mathbf{(c)}$ If $C\setminus \gamma$ is not a circuit of $M,$ and $C\setminus \gamma=C_{np}\cup A,$ for some $np$-circuit $C_{np}$ excluding $e,$ and disjoint from $A \subset E(M),$ where $|A|\geq 2.$ Then $C_{np}\cup \{e, \gamma\}$ is the circuit, of type $\mathcal{C}_4,$ of $M^e_{a,b}$ contained in $C,$ a contradiction.\\
	\noindent $\mathbf{(d)}$ If $C\setminus \gamma$ is not an $np$-circuit of $M,$ and $C\setminus \gamma=C_{np}\cup A,$ for some $np$-circuit $C_{np}$ containing $e,$ and $A\subset E(M).$ If $a \notin C_{np},$ then $(C_{np}\setminus e) \cup \gamma $ is a circuit of type $\mathcal{C}_5$ contained in $C$, a contradiction. And if $a \in C_{np},$ then $C_{np} \cup \gamma, $ a circuit of type $\mathcal{C}_7,$ is contained in $C,$ which is not possible.\\
	Therefore we conclude that $C\in (\mathcal{C}(M'_{a,b}) \cup \mathcal{C}_4 \cup \mathcal{C}_5 \cup \mathcal{C}_6 \cup \mathcal{C}_7 \cup \mathcal{C}_8).$
	
\end{proof}
\begin{corollary}
	The $es$-splitting matroid $M^e_{a,b}$ is not a bipartite matroid.
\end{corollary}
\begin{proof}
	The proof follows from the fact, the $es$-splitting matroid $M^e_{a,b}$ contains the circuit $\Delta=\{e,z,\gamma\}$ of size three.
\end{proof}
\noindent The following example lists all the type of circuits of a $p$-matroid $M,$ and the corresponding $es$-splitting matroid $M^e_{a,b}.$

\begin{example}\label{ex1}
	Consider the ternary sparse paving matroid $P_8$ which has the representation given by matrix $A.$
	
	\begin{center}
		
		$\mathbf{A} =
		\begin{pNiceMatrix}%
		[first-col,
		first-row,
		code-for-first-col = \color{blue},
		code-for-first-row = \color{blue}]
		& 1 & 2 & 3 & 4 & 5 & 6 & 7 & 8   \\
		&  1 & 0 & 0 & 0 & 0 & 1 & 1 & 2  \\
		&  0 & 1 & 0 & 0 & 1 & 0 & 1 & 1  \\
		&  0 & 0 & 1 & 0 & 1 & 1 & 0 & 1  \\
		&  0 & 0 & 0 & 1 & 2 & 1 & 1 & 0  \\
		\end{pNiceMatrix} \qquad
		\mathbf{A^e_{1,4}} =
		\begin{pNiceMatrix}%
		[first-col,
		first-row,
		code-for-first-col = \color{blue},
		code-for-first-row = \color{blue}]
		&  1 & 2 & 3 & 4 & 5 & 6 & 7 & 8 & 9 & 10  \\
		&  1 & 0 & 0 & 0 & 0 & 1 & 1 & 2 & 0 & 0   \\
		&  0 & 1 & 0 & 0 & 1 & 0 & 1 & 1 & 0 & 0   \\
		&  0 & 0 & 1 & 0 & 1 & 1 & 0 & 1 & 0 & 0   \\
		&  0 & 0 & 0 & 1 & 2 & 1 & 1 & 0 & 0 & 1   \\
		&  1 & 0 & 0 & 1 & 0 & 0 & 0 & 0 & 1 & 0   \\
		\end{pNiceMatrix}$
	\end{center}
	\bigskip
	\noindent For $a=1$, $b=4,$ $e=4$ and $\alpha=1$ the representation of $es$-splitting matroid $M^e_{1,4}$ over $GF(3)$ is given by the matrix $A^e_{1,4}$.\\
	The collection of circuits of $M,$ $M_{1,4}$ and $M^e_{1,4}$ is given in the following table.\\

\bigskip
		\begin{longtable}{ | m{4.4cm}| m{4.5cm} | m{5.5cm} |}
			\hline
			\textbf{~~~~~~~~~Circuits of $M$}  & \textbf{~~~~~~~~Circuits of $M_{1,4}$} & \textbf{~~~~~~~~Circuits of $M^e_{1,4}$}\\
			\hline
			$\{1, 2, 4, 5, 6\},$ $\{1, 2, 4, 6, 8\},$ $\{1, 3, 4, 5, 7\},$ $\{1, 3, 4, 7, 8\},$ $\{1, 4, 5, 8\},$ $\{2, 3, 6, 7\},$ $\{2, 5, 6, 8\},$~~~~~~~~~ $\{3, 5, 7, 8\}$& 	$\{1, 2, 4, 5, 6\},$ $\{1, 2, 4, 6, 8\},$ $\{1, 3, 4, 5, 7\},$ $\{1, 3, 4, 7, 8\},$ $\{1, 4, 5, 8\},$ $\{2, 3, 6, 7\},$ $\{2, 5, 6, 8\},$~~~~~~~~ $\{3, 5, 7, 8\}$ & 	$\{1, 2, 4, 5, 6\},$ $\{1, 2, 4, 6, 8\},$ $\{1, 3, 4, 5, 7\},$ $\{1, 3, 4, 7, 8\},$ $\{1, 4, 5, 8\},$ $\{2, 3, 6, 7\},$ $\{2, 5, 6, 8\},$~~~~~~~~~~~~~~ $\{3, 5, 7, 8\}$ \\
			
			\hline

			$\{1, 2, 3, 5, 6\},$ $\{1, 2, 3, 5, 7\},$ $\{1, 2, 3, 8\},$ $\{1, 2, 4, 7\},$ $\{1, 2, 5, 7, 8\},$ $\{1, 2, 6, 7, 8\},$ $\{1, 3, 4, 6\},$ $\{1, 3, 5, 6, 8\},$ $\{1, 3, 6, 7, 8\},$ $\{1, 5, 6, 7\},$ $\{2, 3, 4, 5\},$ $\{2, 3, 4, 6, 8\},$ $\{2, 3, 4, 7, 8\},$ $\{2, 4, 5, 6, 7\},$ $\{2, 4, 5, 7, 8\},$ $\{3, 4, 5, 6, 7\},$ $\{3, 4, 5, 6, 8\},$ ~~~~~~$\{4, 6, 7, 8\}$ & $-$ &  $\{1, 2, 3, 5, 6, 9\},$ $\{1, 2, 3, 5, 7, 9\},$ $\{1, 2, 3, 8, 9\},$ $\{1, 2, 4, 7, 9\},$ $\{1, 2, 5, 7, 8, 9\},$ $\{1, 2, 6, 7, 8, 9\},$ $\{1, 3, 4, 6, 9\},$ $\{1, 3, 5, 6, 8, 9\},$ $\{1, 3, 6, 7, 8, 9\},$ $\{1, 5, 6, 7, 9\},$ $\{2, 3, 4, 5, 9\},$ $\{2, 3, 4, 6, 8, 9\},$ $\{2, 3, 4, 7, 8, 9\},$ $\{2, 4, 5, 6, 7, 9\},$ $\{2, 4, 5, 7, 8, 9\},$ $\{3, 4, 5, 6, 7, 9\},$ $\{3, 4, 5, 6, 8, 9\},$ ~~~~~~$\{4, 6, 7, 8, 9\}$ \\
			
			\hline
			
			$-$ & $-$ & $\{1, 2, 3, 4, 8, 10\},$~~~ $\{1, 4, 5, 6, 7, 10\}$\\
			
			\hline
			$-$ & $-$ & $\{2, 3, 5, 10\},$ $\{2, 3, 6, 8, 10\},$ $\{2, 3, 7, 8, 10\},$ $\{2, 5, 6, 7, 10\},$ $\{2, 5, 7, 8, 10\},$ $\{3, 5, 6, 7, 10\},$ $\{3, 5, 6, 8, 10\},$~~~~~~~~ $\{6, 7, 8, 10\}$\\
			
			\hline
			$-$ & $-$ & $\{1, 2, 5, 6, 9, 10\},$ $\{1, 2, 6, 8, 9, 10\},$ $\{1, 3, 5, 7, 9, 10\},$~~~~~$\{1, 3, 6, 9, 10\},$ $\{1, 3, 7, 8, 9, 10\},$~~~~~ $\{1, 5, 8, 9, 10\}$\\
			\hline
			
			$-$ & $-$ & $\{1, 2, 4, 7, 10\},$~~~~~~~ $\{1, 3, 4, 6, 10\}$\\
			\hline
			$-$ & $-$ & $\{1, 2, 7, 9, 10\}$\\
			\hline
			$-$ & $-$ & $\{4, 9, 10\}$\\
			\hline
		\end{longtable}	

\end{example}
\noindent Consider subsets of $E$ of the type $C\cup I$ where $C=\{u_1,u_2,\ldots,u_l\}$ is an $np$-circuit of $M,$ which is disjoint from an independent set $I=\{v_1,v_2,\ldots,v_k\}$ and $\{a,b\}\subset (C\cup I)$. We say $C\cup I$ is $p$-dependent if it contains no member of $\mathcal {C}_0$ and there are non-zero constants $\alpha_1,\alpha _2,\ldots,\alpha_l$ and $\beta_1, \beta_2,\ldots,\beta_k$ such that $\sum_{i=0}^{l}\alpha_i u_i + \sum_{j=0}^{k}\beta_j v_j = 0 (mod \: p)$ and $coeff.(a)+coeff.(b) = 0(mod \: p).$\\
\noindent We use $cl$ and $cl'$ to denote closure operators on $M$ and $M^e_{a,b},$ respectively.

\begin{theorem}\label{basis}
	Let $M$ be a $p$-matroid and $M^e_{a,b}$ be es-splitting matroid. $\mathcal{B}(M^e_{a,b})=\mathcal{B}(M'_{a,b})\cup \mathcal{B}_1\cup \mathcal{B}_2\cup \mathcal{B}_3,$ where
	\begin{description}
		\item $\mathcal{B}_1=\{B\cup \gamma:$ $B\in \mathcal{B}(M)$ and $e\in B$\};
		\item $\mathcal{B}_2=\{B\cup \gamma:$ $B\in \mathcal{B}(M),$ $a\in B$ and $e\notin$ cl$(B\setminus a)$ \};
		\item $\mathcal{B}_3=\{(C_{np}\cup I)\cup \gamma:$ $(C_{np}\cup I)$ is not a $p$-dependent set of $M,$ $rank(C_{np} \cup I)=rank(M)-1,$ $e\notin cl(C_{np} \cup I)$\};
		\item $\mathcal{B}_4=\{I\cup\{z,\gamma\}:$ $I\in \mathcal{I}(M),$ $e \notin cl(I),$$|I|=r(M)-1$ \}.
	\end{description}
\end{theorem}
\begin{proof}
	Let $B'\in \mathcal{B}(M^e_{a,b}).$ Then $|B'|=r(M)+1.$\\
	$\mathbf{Case~1:}$ $z,\gamma \notin B'.$ Then $B' \in \mathcal{B}(M_{a,b}).$\\
	$\mathbf{Case~2:}$ $z \in B',$$ \gamma \notin B'.$ Then $B' \in \mathcal{B}(M'_{a,b}).$\\
	$\mathbf{Case~3:}$ $z \notin B',$$ \gamma \in B'.$ Then $B'\setminus\gamma$ is an independent set of $M'_{a,b}.$ Denote $B'\setminus\gamma=B.$ By Lemma 2.6 of \cite{elt}, following two subcases are possible:\\
	$\mathbf{Subcase~(i):}$ $B \in \mathcal{I}(M).$ Since $|B|=r(M), B\in \mathcal{B}(M).$ Note that $a\in B$ or $e \in B.$ Otherwise $\gamma \in cl'(B),$ which is a contradiction to the fact  that $B'$ is basis of $M^e_{a,b}.$\\
	$\mathbf{(a)}$ $e \in B.$ Then $B'\in \mathcal{B}_1.$\\
	$\mathbf{(b)}$ $e\notin B.$ If $e \notin cl(B\setminus a),$ then $B'\in \mathcal{B}_2.$ Suppose $e\in cl(B\setminus a),$ there are column vectors in $B\setminus a,$ say $v_1,v_2, \ldots, v_j,$ such that $\{e,v_1,v_2, \ldots, v_j\}$ forms a circuit in $M.$ Since the last coordinate of  $v_1,v_2, \ldots, v_j,$ and $\gamma$ in $M^e_{a,b}$ are zero, and the coordinates of $e$ and the corresponding coordinates of $\gamma$ are identical implies $\{\gamma, v_1,v_2, \ldots, v_j\}$
	is dependent in $M^e_{a,b}.$ Thus $B'$ contains a dependent set, which is a contradiction. \\
	\noindent $\mathbf{Subcase~(ii):}$ $B = (C_{np}\cup I),$ where $(C_{np}\cup I)$ is not $p$ -dependent and it contains no union of two disjoint $np$-circuits.
	Let $e\in(C_{np}\cup I).$ Then either $e\in C_{np}$ or $e\in I.$ Let $e \in C_{np}.$ If $a$ is also in $C_{np},$ then $C_{np} \cup \gamma$ is a  a circuit, of type $ \mathcal{C}_7,$ contained in $B',$ a contradiction. Consider $a \notin C_{np.}$ Then $(C_{np}\setminus e)\cup \gamma$ is a circuit, of type $ \mathcal{C}_5,$ contained in $B',$ a contradiction. Therefore $e \notin C_{np}.$ If $e \in I,$ then $C_{np} \cup \{e,\gamma\}$ is a circuit, of type $ \mathcal{C}_4,$ contained in $B',$ again a contradiction. Therefore $e \notin I.$ Furthermore, we claim that $e \notin cl(C_{np} \cup I).$ On the contrary suppose, $e \in cl(C_{np} \cup I).$ Then $(C_{np} \cup I) \cup e$ contains a circuit, say $C,$ and $e \in C.$ If $a \notin C,$ then $(C\setminus e) \cup \gamma$ is a dependent set contained in $B',$ which is a contradiction. Suppose $a \in C.$ Note that $a$ is also in $ C_{np}.$ Therefore, there exists a circuit $C' \subset (C \cup C_{np})-a.$ Note that $C'$ is not a $p$-circuit. Thus, $e \in C',$ and hence $(C'\setminus e) \cup \gamma $ is a dependent set in $B',$ a contradiction again. Therefore $e \notin cl(C_{np} \cup I).$ Thus, $B' \in \mathcal{B}_3.$ \\
	$\mathbf{Case~4:}$ $z \in B',$$ \gamma \in B'.$ As earlier denote $B'\setminus\gamma=B.$ Note that $e\notin B',$ otherwise $B'$ contains the circuit $\{e,z,\gamma\}.$ Observe that  $B $ is independent in $M'_{a,b}.$ By Lemma 2.6 of \cite{elt}, following three are possible subcases:\\
	$\mathbf{Subcase~(i):}$ $ B = I \cup z,I \in \mathcal{I}(M)$ and $|I|=r(M)-1.$ In this case, $B'\in \mathcal{B}_4.$ \\
	$\mathbf{Subcase~(ii):}$ $B \in \mathcal{I}(M).$ Thus $B'= B \cup \gamma.$ In this case, $B'\in \mathcal{B}_1$ or $B'\in\mathcal{B}_2$\\
	$\mathbf{Subcase~(iii):}$ $B = C_{np} \cup I,$  where $(C_{np}\cup I)$ is not $p$ -dependent and it contains no union of two disjoint $np$-circuits. Then $C_{np}\cup z$ is a circuit of $M^e_{a,b}$ contained in $B',$ a contradiction.
	Therefore $B'\in (\mathcal{B}(M'_{a,b})\cup \mathcal{B}_1\cup \mathcal{B}_2\cup \mathcal{B}_4).$\\
	\noindent Conversely, let $X\in (\mathcal{B}(M'_{a,b})\cup \mathcal{B}_1\cup \mathcal{B}_2\cup \mathcal{B}_3).$ We prove that $X$ is a basis of $M^e_{a,b}.$\\
	If $X\in \mathcal{B}(M'_{a,b}),$ then $|X|=r(M)+1.$ Moreover $X$ is independent in $M^e_{a,b},$ therefore $X$ is basis of $M^e_{a,b}.$ If $X \in \mathcal{B}_1,$ then $X=B\cup \gamma,$ for some $B\in \mathcal{B}(M)$ containing $e.$ Let $B = \{u_1,u_2,\ldots,u_l\}.$ Note that $B$ is independent in $M^e_{a,b}.$ If $X=B\cup \gamma$ is dependent in $M^e_{a,b},$ then it contains a circuit, say $C,$ and $\gamma \in C.$ Following are the three possibilities regarding $C$:\\
	$\mathbf{Case ~1}$ $C= B\cup \gamma =  \{u_1,u_2,\ldots,u_l,\gamma\}.$ Since $e\in B,$ then $a$ must be in $B.$ Assume $u_1=a, u_2=e,$ WLOG. There exists non-zero scalars $\alpha_1,\alpha_2,\ldots,\alpha_l,\alpha_{\gamma},$ such that $\alpha_1 u_1+ \alpha_2 u_2+ \ldots+\alpha_l u_l+ \alpha_{\gamma} u_{\gamma} \equiv 0(mod\:p).$ Since every coordinate of $e$ is equal to the corresponding coordinate of $\gamma,$ restricting the column vectors of $B$ to the matrix $A,$ we get $\alpha_1 u_1+ (\alpha_2 + \alpha_{\gamma}) u_2+ \ldots+\alpha_l u_l  \equiv 0(mod\:p).$ Therefore $B= \{u_1,u_2,\ldots,u_l\}$ is a dependent set of $M,$ a contradiction.\\
	$\mathbf{Case ~2}$ $C \subset B \cup \gamma,$ and $a,e \in C \setminus \gamma.$ Using the arguments as in $\mathbf{(a)},$ we get a contradiction. \\
	$\mathbf{Case ~3}$ $C \subset B \cup \gamma,$ and $a,e \notin C .$ Using the arguments as in $\mathbf{(a)},$ we get a contradiction.\\
	Next, assume $X \in \mathcal{B}_2.$ Then $X=B\cup \gamma,$ where $B\in \mathcal{B}(M),$ $a\in B$ and $e\notin$ cl$(B\setminus a).$ Note that $B$ is independent in $M^e_{a,b}.$ If $X=B\cup \gamma$ is dependent in $M^e_{a,b},$ then it contains a circuit, say $C,$ and $\gamma \in C.$ Then $C \in \mathcal{C}_5$ or $C \in \mathcal{C}_7.$ Since $e\notin cl(B\setminus a),$ and hence $ e\notin C.$ Therefore $C$ cannot be a member of  $\mathcal{C}_5$ and $\mathcal{C}_7.$ \\
	Now if $X \in \mathcal{B}_3,$ then $X = (C_{np}\cup I)\cup \gamma $ for some $np$-circuit $C_{np},$ independent set $I$ of $M.$ Note that $|(C_{np} \cup I)|=rank(M).$ We claim $(C_{np} \cup I)$ is a independent set in $M^e_{a,b}$\\
	\\
	\\
	 and $(C_{np}\cup I)$ is not a $p$-dependent set of $M,$ implies $|(C_{np}\cup I)|=rank(M),$ and it is independent in $M^e_{a,b}.$ Therefore $|X|=rank(M)+1.$ Here it is enough to prove that $X=(C_{np}\cup I)\cup \gamma$ is independent in $M^e_{a,b}.$ On the contrary assume that $X$ is not independent in $M^e_{a,b}.$ Then $X$ contains a circuit, say $C,$ passing through $\gamma.$ Note that $e,z \notin C.$  Therefore $C$ is a circuit of type $\mathcal{C}_5.$ This implies $e \in cl(C) \subset cl(C_{np}\cup I),$ a contradiction again. \\
	Let $X \in \mathcal{B}_4.$ Then $X=I \cup \{z,\gamma\},$ where $I \in \mathcal{I}(M).$ By Lemma 2.6 of \cite{elt} $I'=I \cup z$ is an independent set of $M'_{a,b},$ and hence of $M^e_{a,b}.$ Therefore it is enough to show that $I' \cup \gamma $ is an independent set of  $M^e_{a,b}.$ If not, then $I' \cup \gamma$ contains a circuit $C,$ such that $\gamma \in C.$ Since $e \notin X,$ $C$ cannot be a circuit of the type $\mathcal{C}_4,\mathcal{C}_7$ and $\Delta=\{e,z, \gamma\}.$ Assume $C \in \mathcal{C}_5.$ Then $C=(C'\setminus e)\cup \gamma,$ where $C'$ is an $np$-circuit containing $e,$ and $C'\subset I.$ Thus $e \in cl(I),$ a contradiction. We get the same contradiction, if $C \in \mathcal{C}_6, \mathcal{C}_8.$ Therefore $I'\cup \gamma $ is an independent set of size $rank(M)+1.$
\end{proof}	
\noindent The next theorem gives the rank function of a $es$-splitting matroid $M^e_{a,b}$ in terms of the rank function of corresponding $p$-matroid $M.$
	\begin{theorem}
		Let $r$ and $r'$ denote the rank functions of the $p$-matroids $M$ and $M^e_{a,b}$, respectively. Let $X \subseteq E(M).$ Then

			\begin{equation}\label{r1}
				r'(X)= \begin{cases}
					r(X)+1, \text{if}~ X ~ \text{contains an}~ np-\text{circuit;}~~~ \\
					r(X), \text{otherwise .}
					
				\end{cases}
			\end{equation}

            \begin{equation}\label{r2}
              r'(X\cup z) = r(X)+1~~~~~~~~~~~~~~~~~~~~~~~~~~~~~~~~~~~~~~~~~~
            \end{equation}

           \begin{equation}\label{r3}
           	 ~~~~~~~~~~~~~~~~~~~~~~~~r'(X~\cup ~\gamma) = \begin{cases}
           	 	r(X), ~\text{if}~~ a, e\notin X, \text{and}~ e\in cl(X);~ \\
           	 	r(X)+2, \text{if}~ e\notin cl(X)\text{and}~ X \text{ contains an np circuit;}\\
           	 	r(X)+1, \text{otherwise.}

           	 \end{cases}
           \end{equation}

	        \begin{equation}\label{r4}
	        	r'(X~\cup ~\{z,\gamma\}) = \begin{cases}
	        	r(X)+1, \text{if}~ e\in cl(X);\\
	        	r(X)+2, \text{if}~ e\notin cl(X);~~~~~~~~~~
	        	
	        	\end{cases}
	        \end{equation}
	
	\end{theorem}

   \begin{proof}
    Equation (\ref{r1}) and (\ref{r2}) are proved in \cite{mjg} and \cite{elt}.\\
    \noindent To prove (\ref{r3}), consider $a, e\notin X$ and $e\in cl(X)$ then $X\cup e $ contains an $np$-circuit of $M$, say $C_{np}.$ Then $e \in C_{np}.$  Now $C_{np}\setminus e \subset X.$ In $X\cup \gamma,$$ C_{np}\setminus e$ forms a circuit, of type $\mathcal{C}_5,$ with $\gamma.$ Therefore $\gamma \in cl'(X).$ Thus $r'(X\cup \gamma)=r(X).$  \\
    \noindent Next, assume that $e\notin cl(X)$ and $X$ contains an $np$-circuit of $M$, say $C_{np}.$ Now by \ref{r1}, $r'(X)=r(X)+1.$ We claim that $\gamma \notin cl'(X).$ On the contrary assume that $\gamma \in cl'(X).$ Then there exists a circuit $ C \subset X \cup \gamma$ in  $M^e_{a,b},$ containing $\gamma.$ Such a circuit must be of the type $\mathcal{C}_4,\mathcal{C}_5$ or $\mathcal{C}_7.$ Therefore $e \in cl(C\setminus\gamma),$ that implies $e \in cl(X),$ a contradiction. Therefore $r'(X\cup \gamma) =(r(X)+1)+1= r(X)+2.$\\
   \noindent Suppose  $X$ contains an $np$-circuit $C_{np},$ and $e \in cl(X).$ Then there exists a circuit,say $C,$ containing $e.$ Such a circuit could be $p$ or $np$ circuit. In case when $C$ is an $np$-circuit, containing $e$ then we get a circuit of type $\mathcal{C}_5$ or $\mathcal{C}_7$ of $M^e_{a,b}$. Therefore $r'(X\cup \gamma)=r(X)+1.$  If $C$ is a $p$ circuit, then $\gamma$  cannot form a circuit with $C,$ since for such a circuit to form we need $z\in X\cup \gamma.$ Thus $\gamma \notin cl'(X)$ in $M^e_{a,b}.$ Therefore $r'(X\cup \gamma)=r(X)+1.$ \\
   \noindent Next, assume that $X$ contains no $np$-circuit in $M,$ and $e\in cl(X).$ If $X\cup e$ contains an $np$-circuit, then as discussed earlier $r'(X\cup \gamma)=r(X).$ If $X\cup e$ contains a $p$-circuit, then $\gamma \notin cl'(X),$ as $z \notin X\cup \gamma.$ Therefore $r'(X\cup \gamma)=r(X)+1.$ Further if $e \notin cl(X),$ then $\gamma \notin cl'(X).$ Therefore $r'(X\cup \gamma)=r(X)+1.$ \\
   \noindent To prove (\ref{r4}), consider $e \in cl(X)$ in $M.$ Note that $r'(X \cup z)=r(X)+1.$ Since $e \in cl(X), X \cup\{z,\gamma\}$ contains a circuits of the type $\Delta, \mathcal{C}_6,$ or $\mathcal{C}_8.$ Thus $ r'(X \cup\{z,\gamma\})=r'( X \cup z)=r(X)+1.$ If $e \notin cl(X),$ then  $ \gamma \notin cl'(X\cup z);$ otherwise there exist a circuit $C\subset X\cup z,$ containing $\gamma,$ for such a circuit to exists $e$ must be in $cl(C\setminus\{z,\gamma\})$ in $M.$ Therefore $e\in X,$ a contradiction. Thus $r'(X\cup \{z,\gamma\})=r(X)+2.$\\
\end{proof}
  \noindent Equation (\ref{r3}) and equation (\ref{r4}) can be verified with Ex.\ref{ex1} as follows:
  \begin{enumerate}
  	\item $r'(\{2,3,5\}\cup 10)=3=r(\{2,3,5\}).$
  	\item $r'(\{1,2,3,8\}\cup 10)=5=r(\{1,2,3,8\})+2.$
  	\item $r'(\{1,2,3,5,6\}\cup 10)=5=r(\{1,2,3,5,6\})+1.$
  	\item $r'(\{1,2,4,7\}\cup \{9,10\})=4=r(\{1,2,4,7\})+1.$
  	\item $r'(\{1,2,3,8\}\cup \{9,10\})=5=r(\{1,2,3,8\})+2.$
  \end{enumerate}
    
 \begin{corollary}
 In the $es$-splitting matroid $M^e_{a,b},$  $\{a,e,z\}$ forms a cocircuit.
 \end{corollary}
\begin{proof}
	Let $X = E(M^e_{a,b})\setminus \{a,e,z\}.$ We will show that $X$ is a hyperplane of $M^e_{a,b}.$ First note that the last coordinate of every vector in $X$ is zero and that of $a,e,$ and $z$ is $1.$ Therefore, no linear combination of vectors in $X$ yields $a,e,$ or $z.$ Thus, $a,e,z \notin cl'(X).$ Equivalently $cl'(X)= X.$ Next, we show that $r'(X)=r(M).$ As noted earlier that every vector in $X$ has last coordinate equal to zero, and every coordinate of $\gamma$ is equal to the corresponding coordinate of $e,$ therefore $X= E(M)\setminus a,$ and $r'(X)=r(X).$ Since $M$ is a coloopless, there exist a basis $B$ of $M$ such that $a \notin B.$ Then  $r'(X)=r(X)=|B|=r(M).$ Therefore $X$ is a hyperplane of $M^e_{a,b}.$\
	
	
\end{proof}

\section{Connectivity of $es$-splitting $p$-matroids}
Connectivity of a matroid is an important property in matroid theory. In particular, connectivity plays an significant role in the proof of Rota's conjecture \cite{GGW1, GGW2}. We recall the following lemma in this regard. 

\begin{lemma}\cite{GGW1} For each field $\mathcal F,$ each excluded minor for the class of $\mathcal F$-representable matroids is 3-connected.\end{lemma} 

The es-splitting operation on an $n$-connected binary matroid may not yield an $n$-connected matroid for $(n \geq 3)$. Malavadkar et al. \cite{mdsncm}, proved that given an $n$-connected binary matroid $M$ of rank $r$, the resulting es-splitting binary matroid has an $n$-connected minor of rank-$(r+1)$ with $|E(M)|+1$ elements.

In this section, we show that the $es$-splitting operation on $p$-matroids preserves connectivity and 3-connectedness.
Let $M$ be a matroid having ground set $E,$ and $k$ be a positive integer. The $k$-separation of matroid $M$ is a partition $\{S, T\}$ of $E$ such that $|S|, |T|\geq k$ and $r(S) + r(T)- r(M) <k.$ For an integer $n \geq 2,$ we say $M$ is $n$-connected if $M$ has no $k$- separation, where $1 \leq k \leq n-1.$ 

\begin{lemma}
	If $M$ is a connected $p$-matroid, then the corresponding  $es$-splitting matroid $M^e_{a,b}$ is also connected. 
\end{lemma}
\begin{proof}
	As $M$ is a connected $p$-matroid, by Theorem 3.1 of \cite{elt} the element splitting matroid $M'_{a,b}$ is also connected. So it is enough to prove that for every element $x\in E\cup z,$ there exists a circuit passing through $x$ and $\gamma.$\\
	$\mathbf{Case~1:}$ Suppose $x=z.$ Then $\Delta=\{e, z, \gamma\}$ is the required circuit.\\
	$\mathbf{Case~2:}$  Suppose $x\in E.$ If $x=e,$ then the circuit $\Delta=\{e, z, \gamma\}$ contains $x$ and $\gamma.$ Now assume $x \notin e,$ since $M$ is a connected matroid, there exist a circuit, say $C,$ passing through $x$ and $e.$ Such a circuit $C$ is a $p$-circuit or $np$-circuit. If $C$ is a $p$-circuit, then $(C\setminus e) \cup \{z,\gamma\}$ is the desired circuit passing through $x$ and $\gamma.$ Next assume $C$ is an $np$-circuit. Then the desired circuit passing through $x$ and $\gamma$ will be a member of one of the collection $\mathcal{C}_5,$ $\mathcal{C}_7$ or $\mathcal{C}_8.$ 
\end{proof}

 \noindent We use the following proposition to prove the next theorem:
 \begin{proposition}\cite{ox}
 If a matroid $M$ is $n$-connected matroid and $|E(M)|>2(n-1)$, then all circuits and all cocircuits of $M$ have at least $n$ elements.
 \end{proposition}
 
\begin{theorem}
	Let $M$ be a 3-connected $p$-matroid and $|E(M)|>4$. Then $M^e_{a,b}$ is also a 3-connected matroid.
\end{theorem}
\begin{proof}
	We will show that $M^e_{a,b}$ cannot have a 1 and 2-separation. As $M^e_{a,b}$ is connected, it cannot have a 1-separation. Next assume that $M^e_{a,b}$ has a 2-separation $X,Y$ where $X,Y\subset E\cup\{z,\gamma\},$ $|X|,|Y|\geq 2,$ and 
	\begin{equation}\label{eq1}
		r'(X)+r'(Y)-r'(M^e_{a,b})<2
	\end{equation}
$\mathbf{Case~1:}$ $ z \in X$ and $\gamma \in Y$ \\
Let $X'=X\setminus z$ and $Y'=Y\setminus \gamma$\\
\noindent$\mathbf{Subcase~(i):}$ $|X|=2$ and $X=\{x,z\},$ where $x\in E(M).$ Then by equation (\ref{r2}), $r'(X)=r(X')+1.$ As $M$ is $3$-connected, no two element of $M$ are in series, which guarantees that there is a circuit containing $a$ or $e$ but not both. Such a circuit is an $np$-circuit in $Y$ and hence $r'(Y)=r(Y')+1,$ by equation (\ref{r1}) Therefore equation (\ref{eq1}) becomes $r(X')+r(Y')-r(M)<1,$ a contradiction to the fact that $M$ is $3$-connected.\\
$\mathbf{Subcase~(ii):}$ $|Y|=2$ and $Y=\{x,\gamma\},$ where $x\in E(M).$ Observe that $r'(X)= r(X')+1$ as $z\in X$, and by equation (\ref{r3}) $r'(Y)= r(Y')+1.$ In this case equation (\ref{eq1}) becomes $r(X')+r(Y')-r(M)<1,$ which implies $(X',Y')$ forms a $1$-separation of $M.$ A contradiction again.\\
$\mathbf{Subcase~(iii):}$ $|X|,|Y|>2$\\
Then $|X'|, |Y'|\geq 2$ and by equation (\ref{r2}) $r'(X) = r(X')+1,$ and by equation (\ref{r3}) $r'(Y)\geq r(Y').$ Therefore equation (\ref{eq1}) yields, $r(X')+r(Y')-r(M) < 2.$ That is $(X',Y')$ forms a $2$-separation of $M;$ a contradiction.\\
$\mathbf{Case~2:}$ $\{z,\gamma\}\subseteq X$\\
Let $X'=X\setminus\{z,\gamma\}$\\
$\mathbf{Subcase~(i):}$  $X=\{z,\gamma\}.$ Then $Y = E(M).$ Note that $r'(X)=2,$ and $r(Y)=r(M).$ Now equation (\ref{eq1}) gives, $2+ (r(Y)+1)-(r(M)+1)<2.$ Therefore, $r(Y)<r(M),$ which is not possible. \\
 $\mathbf{Subcase~(ii):}$ $X=\{z,\gamma, x\}$ where $x\in E(M)$ \\
 Here $X'=\{x\}$ we have\\
 \begin{center}
 	 $r'(X) =  \left\{ \begin{array}{rcl}
 	r(X')+1, & \mbox{~~if~$e \in cl(X')$}
 	\\ r(X')+2, & \mbox{~~if~$e \notin cl(X')$} 
 	\end{array}\right. $
 \end{center}
If $x=e,$ then $r'(X)=r(X')+1.$ Again using the fact that there is an $np$-circuit excluding $e$ contained in $Y,$ we have $r'(Y)=r(Y)+1.$ Thus equation (\ref{eq1}) yields $r(X')+r(Y)-r(M)<1$ which means $(X', Y)$ forms a $1$-separation of $M,$ a contradiction.\\
If $x\neq e,$ then $r'(X)=r(X')+2,$ and $r'(Y)=r(Y')+1.$ Again equation (\ref{eq1}) yields $r(X')+r(Y)-r(M)<0,$ which is a contradiction.\\
 $\mathbf{Subcase~(iii):}$ $|X|>3$\\
 Here $|X'|,|Y|\geq 2.$ If $Y$ contains an $np$-circuit of $M,$ then, by equation (\ref{r1}) $r'(Y)=r(Y)+1,$ and by equation (\ref{r4}) $r'(X)\geq r(X')+1.$ Therefore equation (\ref{eq1}) gives $r(X')+r(Y)-r(M)<1,$ which means $(X',Y)$ forms a $1$-separation of $M;$ a contradiction.\\
 If $Y$ contains no $np$-circuit of $M,$ then $r'(Y)=r(Y)$ and $r'(X)\geq r(X')+1.$ Again equation (\ref{eq1}) yields $r(X')+r(Y)-r(M)<2,$ implying $(X', Y)$ is a $2$-separation of $M,$ a contradiction.\\
 So we conclude that $M^e_{a,b}$ cannot have such $2$-separation. Hence $M^e_{a,b}$ is $3$-connected.
\end{proof}

\section{Applications}

The present section characterizes Eulerian $es$-splitting $p$-matroids $M^e_{a,b},(p>2).$ Also we provide a sufficient condition to yield Hamiltonian $p$-matroids from  Hamiltonian $p$-matroids under $es$-splitting operation.

\begin{theorem}\label{EU}
	Let $M$ be a $p$-matroid $(p>2),$ and $a,b \in E(M).$ Then the $es$-splitting matroid $M^e_{a,b}$ is Eulerian if and only if there exists a collection of circuits $\{ C_{np},C_1,\ldots,C_k \},$ where $C_{np}$ is an $np$-circuit containing $a,b,$  $E(M)= C_{np} \cup C_1\cup C_2\cup \ldots \cup C_k,$ and  all the circuits in the collection are pairwise disjoint, except $C_{np} \cap C_1 = \{b\}.$ 
	
\end{theorem}

\begin{proof}
	First assume that $M^e_{a,b}$ is Eulerian $p$-matroid, and $E(M^e_{a,b})=C'_1 \cup C'_2\cup C'_3\ldots \cup C'_k,$ where $C'_1, C'_2, \ldots,C'_k$ are disjoint circuits of $M^e_{a,b}.$ Set $e=b$. Without loss of generality, we assume $\gamma \in C'_1.$\\
	$\mathbf{Case~(i):}$ $C'_1=\Delta=\{e, z, \gamma\}$ and $a\in C'_2.$\\
	Since $\gamma, \notin C'_2,$ $C'_2$ is not a circuit of type $ \mathcal{C}_i; i=4,5,6,7,8. $ Also $C'_2\notin \mathcal{C}_z,$ as $z \notin C'_2.$ No circuits of $M_{a,b}$ contains only $a,$ therefore $C'_2\notin \mathcal{C}(M_{a,b})$.Therefore $C'_2 $ is not a circuit of $M^e_{a,b},$ a contradiction.\\
	\noindent $\mathbf{Case~(ii):}$ $C'_1\in \mathcal{C}_4,$ \\
	Then $e, a, \gamma \in C'_1.$ Assume $z\in C_2'.$ Since $\gamma \notin C'_2,$ $C'_2$ is not a circuit of type $ \mathcal{C}_i; i=4,5,6,7,8$ and $C'_2\neq \Delta.$ Also $C_2'\notin \mathcal{C}_z,$ as both $a,e \in C'_1.$  No circuit of $M_{a,b}$ contains $z,$ therefore $C_2'\notin \mathcal{C}(M_{a,b}).$ Therefore $C'_2 $ is not a circuit of $M^e_{a,b},$ a contradiction.\\
	\noindent By similar arguments as in $\mathbf{Case~(i)},$$\mathbf{Case~(ii)},$ we get a contradiction if $C'_1\in \mathcal{C}_6,$ or $ C'_1\in\mathcal{C}_7,$ or $C'_1\in \mathcal{C}_8.$\\
	\noindent $\mathbf{Case~(iii):}$ $C'_1\in \mathcal{C}_5,$\\
	$C'_1=(C\setminus e)\cup \gamma,$ where $C$ is an $np$-circuit containing $e$ but not $a.$ Thus  $z, e, a\notin C'_1.$ Since $\gamma \notin C'_2,$ $C'_2$ is not a circuit of type $ \mathcal{C}_i; i=4,5,6,7,8,$ and $C'_2\neq \Delta.$ Assume,WLOG, $a \in C'_2.$ Then following are the possible subcases:\\
	\noindent $\mathbf{Subcase~(i):}$ $e,z \in C'_2.$\\
	Observe that $a,e,z \in C'_2.$ As $z \in C_2'$ but $\gamma \notin C_2',$  $C'_2$ is a circuit of type $\mathcal{C}_z.$ Therefore $C'_2\setminus z,$ is an $np$-circuit. Set $C_{np}= C'_2\setminus z.$ Note that $C$ and $C_{np}$ are $np$-circuits of $M$ and $C \cap C_{np}=\{e\}.$  Thus the collection $\{ C, C_{np}, C'_3,\ldots,C'_k \}$ is the desired collection of the circuits of $M,$ where all the circuits are pairwise disjoint, except  $C \cap C_{np}=\{e\}.$  \\
	\noindent $\mathbf{Subcase~(ii):}$ $e\in C'_2, $ but $z \notin C'_2.$\\
	Let $z \in C'_i$ for some $i \in \{ 3,4,\dots,k\}.$ Since $a,e \notin C_i,$ $C_i$ cannot be a circuit of $M^e_{a,b},$ a contradiction.\\
	\noindent $\mathbf{Subcase~(iii):}$ $e,z\notin C'_2.$ \\
	Let $e \in C'_i$ and $z \in C'_j$ for some $i\neq j \in \{ 3,4,\dots,k\}.$ Since $a,e \notin C_j,$ $C_j$ cannot be a circuit of $M^e_{a,b},$ a contradiction.
	
  \noindent Conversely, suppose there exists a collection of circuits $\{ C_{np},C_1,\ldots,C_k \},$ where $C_{np}$ is an $np$-circuit containing $a,b,$  $E(M)= C_{np} \cup C_1\cup C_2\cup \ldots \cup C_k,$ and  all the circuits in the collection are pairwise disjoint, except $C_{np} \cap C_1 = \{b\}.$ Note that the circuits $C_2,C_3, \ldots,C_k $ are $p$-circuits, and $C_{np}$ and $C_1$ are $np$-circuits. Therefore $\{C_2,\ldots,C_k \} \subset \mathcal{C}(M^e_{a,b}).$ Observe that $e \in C_1,$ but $a \notin C_1.$ Thus $C'_1=(C_1 \setminus e) \cup \gamma $ forms a circuit of type $\mathcal{C}_5$ in $M^e_{a,b},$ and $C_{np} \cup z$ forms a circuit of type $\mathcal{C}_z.$ The collection $\{ C_{np} \cup z,C'_1,C_2, C_3,\ldots,C_k \}$ gives a circuit decomposition of $E(M^e_{a,b}).$ Therefore, $M^e_{a,b}$ is Eulerian.
\end{proof}

 \noindent Theorem \ref{EU} can be verified with Eulerian $p$- matroid discussed in Example \ref{ex1}. Here $a=1,b=e=4.$ Take $C_{np}=\{1,2,4,7\}$ and $C_1=\{3,4,5,6,8\}.$ Note that,$1,4\in C_{np},$$ C_{np}\cap C_1=\{4\},$ and $E(M)=C_{np}\cup C_1.$ The $es$-splitting matroid $M^e_{1,4}$ is Eulerian with circuit decomposition $\{\{1,2,4,7,9\},\{3,5,6,8,10\}\}=\{C_{np}\cup z, (C_1\setminus e ) \cup \gamma \}.$\\
 
 \noindent In the next example we verify Theorem \ref{EU} for a non-Eulerian $p$-matroid.
 
 \begin{example} \label{counter}
 	Consider the vector matroid $M \cong M[A]$ represented by the matrix $A$ over field $GF(5)$.
 	
 	\begin{center}
 		$\mathbf{A} = 
 		\begin{pNiceMatrix}%
 			[first-col,
 			first-row,
 			code-for-first-col = \color{black},
 			code-for-first-row = \color{black}]
 			& 1 & 2 & 3 & 4 & 5 \\
 			&  1 & 0 & 0 & 1 & 1\\
 			&  0 & 1 & 0 & 1 & 1 \\
 			&  0 & 0 & 1 & 1 & 0 \\
 			
 		\end{pNiceMatrix} \qquad
 		\mathbf{A^e_{3,5}} = 
 		\begin{pNiceMatrix}%
 			[first-col,
 			first-row,
 			code-for-first-col = \color{black},
 			code-for-first-row = \color{black}]
 			& 1 & 2 & 3 & 4 & 5 & 6 & 7\\
 			&  1 & 0 & 0 & 1 & 1 & 0 & 1\\
 			&  0 & 1 & 0 & 1 & 1 & 0 & 1\\
 			&  0 & 0 & 1 & 1 & 0 & 0 & 0\\
 			&  0 & 0 & 1 & 0 & 1 & 1 & 0\\

 		\end{pNiceMatrix}$
 	\end{center}
 	
 	\noindent The circuits of $M$ are $\{1,2,3,4\},\{1,2,5\},\{3,4,5\}.$ Therefore  $M$ is a non-Eulerian matroid. Consider $es$-splitting matroid $M^e_{3,5} \cong M[A^e_{3,5}]$ corresponding to $a=3$, $b=e=5.$ Take $C_{np}= \{3,4,5\}, C_1= \{1,2,5\}.$ Note that $C_{np}\cap C_1= \{5\},$ and $C_{np}\cup C_1=E(M).$ The resulting $es$-splitting matroid $M^e_{3,5}$ is Eulerian with circuit decomposition $\{\{3,4,5,6\}, \{1,2,7\}\}= \{C_{np}\cup z, (C_1\setminus e) \cup \gamma\}.$
 	
 \end{example}

\begin{proposition}
	Let $M$ be a $p$-matroid containing an $np$-circuit $C_{np}$ of size $r(M)+1.$ Then $M^e_{a,b}$ is Hamiltonian. In this case Hamiltonian circuits in $M^e_{a,b}$ are:
	\begin{enumerate}
		\item $C_{np}\cup z,$
		\item  $C_{np}\cup \gamma,$ if $a,e \in C_{np}.$
	\end{enumerate}

\end{proposition}

\end{document}